\documentclass[11pt]{smfart}
\usepackage{color}
\usepackage{amssymb,verbatim}
\usepackage{hyperref}
\usepackage{amsthm,array,amssymb,amscd,amsfonts,amssymb,latexsym, url}
\usepackage{amsmath}
\usepackage[all]{xy}
\usepackage[french]{babel}

\newcounter{spec}
{\end{list}}

\renewcommand{\P}{{\mathbf P}}

\newcommand{\Z}{{\mathbb Z}}
\newcommand{\Q}{{\mathbb Q}}
\newcommand{\C}{{\mathbb C}}

\renewcommand{\L}{{\mathbb L}}

\newcommand{\ovk}{{\overline k}}

\newcommand{\ovX}{{\overline X}}

\newcommand{\Gal}{{\operatorname{Gal }}}
\newcommand{\Br}{{\operatorname{Br   }}}

\newcommand{\Ker}{{\operatorname{Ker}}}

\renewcommand{\lim}{\varprojlim}

\numberwithin{equation}{section}

\newfont{\gothic}{eufb10}

\renewcommand{\qed}{{\hfill$\square$}}

\newtheorem{theo}{Th\'{e}or\`{e}me}[section]
\newtheorem{prop}[theo]{Proposition}

\newtheorem{lem}[theo]{Lemme}
\newtheorem{cor}[theo]{Corollaire}
\theoremstyle{definition}
\newtheorem{defi}[theo]{D\'efinition}
\theoremstyle{remark}
\newtheorem{rema}[theo]{Remarque}

\newcommand{\bthe}{\begin{theo}}
\newcommand{\ble}{\begin{lem}}
\newcommand{\bpr}{\begin{prop}}
\newcommand{\bco}{\begin{cor}}
\newcommand{\bde}{\begin{defi}}
\newcommand{\ethe}{\end{theo}}
\newcommand{\ele}{\end{lem}}
\newcommand{\epr}{\end{prop}}
\newcommand{\eco}{\end{cor}}
\newcommand{\ede}{\end{defi}}

\newcommand{\Pic}{\operatorname{Pic}}

\newcommand{\F}{{\mathbb F}}

\def\X{{\overline X}}

\def\k{{\overline k}}
\DeclareFontFamily{U}{wncy}{}
\DeclareFontShape{U}{wncy}{m}{n}{%
<5>wncyr5%
<6>wncyr6%
<7>wncyr7%
<8>wncyr8%
<9>wncyr9%
<10>wncyr10%
<11>wncyr10%
<12>wncyr6%
<14>wncyr7%
<17>wncyr8%
<20>wncyr10%
<25>wncyr10}{}
\DeclareMathAlphabet{\cyr}{U}{wncy}{m}{n}

\begin{document}

  \title[Surfaces stablement rationnelles sur un corps quasi-fini] {Surfaces stablement rationnelles sur un corps quasi-fini}

\author{J.-L. Colliot-Th\'el\`ene}
\address
{Universit\'e Paris Sud\\Math\'ematiques, B\^atiment 425\\91405 Orsay Cedex\\France}
\email{jlct@math.u-psud.fr }

\date{soumis le 27 novembre 2017; r\'evis\'e le 14 juin 2018}
\maketitle

 \begin{abstract}
 Sur un corps fini  $\F$,
 toute $\F$-surface stablement $\F$-rationnelle est $\F$-rationnelle.
 Plus g\'en\'eralement, si une $\F$-surface $X$ projective et lisse
 et g\'eom\'etriquement rationnelle n'est pas $\F$-rationnelle, alors
 il existe une extension finie $\F'$ de $\F$ avec $\Br(X_{\F'}) \neq 0$.
 Ceci vaut plus g\'en\'eralement pour une telle surface $X$ sur un corps $k$ quasi-fini  
  d\`es que $X(k) \neq \emptyset$.
  \end{abstract}

\section{Introduction}

L'\'enonc\'e suivant est essentiellement connu  (cf. \cite[\S 1, \S 2]{CTSRequiv}, \cite[\S 2.A]{CTSDesc}).
\begin{theo}\label{implgen} Soit $k$ un corps et $\ovk/k$ une cl\^oture 
s\'eparable  de $k$.
Soit $X$ une $k$-vari\'et\'e projective, lisse.
Soit $\ovX = X \times_{k} \ovk$.
Supposons que  $\ovX$ est $\ovk$-rationnelle, i.e $\ovk$-birationnelle  \`a un espace projectif.
Consid\'erons les conditions suivantes.

(i) La $k$-vari\'et\'e  $X$ est $k$-rationnelle.

(ii) La $k$-vari\'et\'e $X$ est stablement $k$-rationnelle.

(iii) La $k$-vari\'et\'e  $X$ est facteur direct d'une vari\'et\'e $k$-rationnelle,
  c'est-\`a-dire qu'il existe   une $k$-vari\'et\'e $Y$  projective et lisse, g\'eom\'etriquement connexe, telle
que $X \times_{k} Y$ est $k$-birationnelle \`a un espace projectif.

(iv) Le module galoisien   $\Pic(\X)$ est stablement de permutation, c'est-\`a-dire qu'il existe des modules
de permutation de type fini $P_{1}$ et $P_{2}$ et un isomorphisme de modules galoisiens $\Pic(\X) \oplus P_{1} \simeq P_{2}$.

(v) Le module galoisien   $\Pic(\X)$ est un facteur direct d'un module de permutation, c'est-\`a-dire qu'il existe 
un module galoisien $M$, un module de de permutation de type fini $P$ et un 
isomorphisme de modules galoisiens $\Pic(\X) \oplus M \simeq P$.

(vi) Pour toute extension finie s\'eparable $k'/k$, on a $H^1(k', \Pic(\X))=0$.

(vii) Pour toute extension finie s\'eparable $k'/k$, l'application naturelle de groupes de Brauer $\Br(k') \to \Br(X_{k'})$ est surjective.
 
Alors : (i) implique (ii), qui  implique (iii); (ii) implique (iv);  (iii) implique (v); (iv) implique (v);
(v) implique (vi); (vi) implique (vii).
\end{theo}
\begin{proof}
Pour l'implication (i) implique (iv), voir  \cite[Prop. 2A1, p. 461]{CTSDesc}.
Le fait que (ii) implique (iv) et  que (iii) implique (v) r\'esulte alors du calcul du groupe
de Picard d'un produit \cite[Lemme 11]{CTSRequiv}.  L'implication (v) implique (vi)
r\'esulte du fait que pour tout module $P$ galoisien de  permutation et toute
extension finie s\'eparable $k'/k$, on a $H^1(k',P)=0$.
Pour  $X$ une $k$-vari\'et\'e projective, lisse, g\'eom\'etriquement  int\`egre, 
et toute
extension finie s\'eparable $k'/k$,
on a la suite exacte  (cf. \cite[(1.,5.0), p. 386]{CTSDesc}) :
$$ \Br(k') \to \Ker[\Br(X _{k'}) \to \Br(\X)] \to H^1(k', \Pic(\X))).$$
 Si de plus $\ovX$ est $\ovk$-rationnelle, alors
$\Br(\X)=0$. Cette annulation est bien connue en caract\'eristique z\'ero,
et l'argument vaut encore pour la
$\ell$-torsion de $\Br(\X)$ pour
 $\ell$ un nombre premier  premier \`a l'exposant caract\'eristique de $k$.
  Que $\Br(\X)=0$ vaut pour tout corps s\'eparablement clos 
  $\ovk$ et $\X$   $\ovk$-rationnelle
  se voit en combinant  \cite[Cor. 5.8]{GrBr3} et \cite[Prop. 2.1.9]{CTBarbara}.
Ainsi (vi) implique (vii).
\end{proof}
On dit qu'une $k$-surface projective, lisse, g\'eom\'etriquement rationnelle,
est d\'eploy\'ee par une sous-extension galoisienne $K \subset \ovk$
si $X(K) \neq \emptyset$ et si l'inclusion naturelle de r\'eseaux $\Pic(X_{K}) \to \Pic(\ovX)$ est un isomorphisme.
Sous l'hypoth\`ese $X(k) \neq \emptyset$,  ceci \'equivaut au fait
que ${\rm Gal}(\ovk/K)$ agit trivialement sur le r\'eseau $\Pic(\ovX)$.
 
\begin{theo}\label{vraiprincipal}
Soient $k$ un corps et $X$ une $k$-surface projective, lisse, g\'eom\'etriquement rationnelle.
Supposons que $X$ poss\`ede un point $k$-rationnel et que
$X$ soit d\'eploy\'ee par une extension cyclique de $k$.
Si  $X$ n'est pas  $k$-rationnelle, alors  il existe une extension finie 
s\'eparable $k'/k$ telle que $H^1(k', \Pic(\X))\neq 0$, et alors
$X$ n'est pas stablement $k$-rationnelle.
\end{theo}

La d\'emonstration   sera donn\'ee au \S 4 (th\'eor\`eme \ref{vraiprincipalbis}),
o\`u l'on regroupe les r\'esultats du \S 2 (surfaces fibr\'ees en coniques) et  du \S 3 (surfaces de del Pezzo).
C'est une d\'emonstration cas par cas, qui s'appuie de fa\c con essentielle sur  les tables 
\'etablies par divers auteurs
sur l'action du groupe  de Galois sur le groupe de Picard des surfaces de del Pezzo
de degr\'e  3,2,1.

L'\'enonc\'e du th\'eor\`eme  \ref{vraiprincipal}
est \`a comparer avec les exemples donn\'es dans \cite{BCTSaSD}. 
Si $k$ est un corps de caract\'eristique diff\'erente de 2, et  $P(x) \in k[x]$ un polyn\^ome s\'eparable et irr\'eductible
de degr\'e 3, de discriminant $a \in k^{*}$ non carr\'e, on montre que la $k$-surface d'\'equation affine
$ y^2-az^2=P(x)$
est stablement $k$-rationnelle mais non $k$-rationnelle.
Il existe donc de telles surfaces sur tout 
  corps $ k$ de caract\'eristique diff\'erente de 2
poss\'edant une extension de corps galoisienne de groupe $\frak{S}_{3}$, 
par exemple le corps
des rationnels ou le corps $F=\C(t)$ des fractions rationnelles en une variable sur les complexes.

Par d\'efinition, un corps quasi-fini est un corps parfait dont la cl\^oture galoisienne
est le groupe procyclique $\hat{\Z}$  \cite[Chap. XIII, \S 2]{serre}. Les deux types d'exemples
classiques sont : les corps finis et les corps de s\'eries formelles d'une variable sur un corps
alg\'ebriquement clos de caract\'eristique z\'ero.

Si $k$ est un corps fini, ou un corps de s\'eries formelles d'une variable $\C((t))$
sur un corps alg\'ebriquement clos $\C$ de caract\'eristique z\'ero,
toute $k$-surface projective, lisse, g\'eom\'etriquement rationnelle
poss\`ede un $k$-point (Proposition \ref{C1}), et toute extension finie de corps
$K/k$ est cyclique. Toute conique projective et lisse sur $k$ est $k$-isomorphe \`a $\P^1_{k}$.

Le th\'eor\`eme \ref{vraiprincipal} et le th\'eor\`eme \ref{implgen}
donnent alors l'\'enonc\'e suivant,
qui pour un corps fini r\'epond \`a une question
de B. Hassett mentionn\'ee par A. Pirutka dans \cite{P}.

\begin{theo}\label{principal}
Soient $k$ un corps et $X$ une $k$-surface projective, lisse, g\'eom\'etriquement rationnelle.
Sous l'une des hypoth\`eses suivantes :

(i) $X$ poss\`ede un $k$-point et le corps $k$ est quasi-fini;

(ii) le corps $k$ est un corps fini;

(iii) le corps $k=\C((t))$ est le corps des s\'eries formelles en une variable sur un corps $\C$  alg\'ebriquement clos de caract\'eristique z\'ero.

on a :

(a) Si pour toute extension finie de corps $k'/k$ on a $H^1(k', \Pic(\X))=0$, alors $X$ est $k$-rationnelle.

 (b) Les conditions (i) \`a (vii)   du th\'eor\`eme \ref{implgen} sont \'equivalentes pour $X$. En particulier, sur un tel corps $k$, toute $k$-surface  stablement $k$-rationnelle est $k$-rationnelle.
\end{theo}

\begin{cor} Soient $k$ un corps et $X$ une $k$-surface projective, lisse, g\'eom\'etriquement rationnelle.
Sous l'une des hypoth\`eses suivantes :

(i) $X$ poss\`ede un $k$-point et le corps $k$ est quasi-fini;

(ii) le corps $k$ est un corps fini;

(iii) le corps $k=\C((t))$ est le corps des s\'eries formelles en une variable sur un corps $\C$  alg\'ebriquement clos de caract\'eristique z\'ero.

Sous  l'une quelconque des trois  hypoth\`eses suivantes :

(a) le pgcd des degr\'es des extensions finies   $K/k$ telles que $X_{K}$ est stablement $K$-rationnelle est \'egal \`a 1.

(b) $X$ est $k$-unirationnelle, et le pgcd des degr\'es des $k$-applications rationnelles dominantes 
g\'en\'eriquement s\'eparables de $\P^2_{k}$ vers $X$ est \'egal \`a 1,

(c) le groupe de Chow des z\'ero-cycles de $X$ est universellement trivial,

\noindent  la surface $X$ est $k$-rationnelle.
\end{cor}

\begin{proof} 
On note $CH_{0}(X)$ le groupe de Chow des classes de z\'ero-cycles sur $X$,
et $A_{0}(X)$ le sous-groupe des classes de z\'ero-cycles de degr\'e z\'ero.
Ces groupes sont des invariants $k$-birationnels des vari\'et\'es projectives, lisses,
g\'eom\'etriquement int\`egres.
Ceci est \'etabli dans \cite[Prop. 6.3]{CTCoray}
avec quelques restrictions, par exemple en caract\'eristique z\'ero. 
La d\'emonstration de  \cite[Example 16.1.11]{Fulton} 
vaut sur un corps quelconque.

L'hypoth\`ese (c) est que pour toute extension de corps $F/k$
l'application degr\'e $deg_{F} : CH_{0}(X_{F}) \to \Z$ est un isomorphisme,
ce qui est le cas pour $X$ projective, lisse, int\`egre, stablement $k$-rationnelle.

En caract\'eristique z\'ero,
d'apr\`es \cite[Prop. 6.4]{CTCoray}, chacune des hypoth\`eses (a) et (b) sur la surface $X$
implique (c). Via les correspondances \cite[Chap. 16]{Fulton} 
on voit que ceci vaut sur un corps quelconque.

 Sous l'hypoth\`ese de (c), le module galoisien $ \Pic(\X)$ est un facteur direct 
d'un module de permutation. Pour la d\'emonstration, je renvoie \`a  \cite[Appendix A]{Gille}.
Pour toute extension  de corps $E/k$,
on a  donc $H^1(E, \Pic(\X))=0$. Le th\'eor\`eme \ref{principal} permet de conclure.
\end{proof}

 Rappelons la classification $k$-birationnelle des $k$-surfaces projectives, lisses,
  g\'eom\'etriquement rationnelles, due \`a Enriques, Manin, Iskovskikh \cite{Isk79}, et Mori.
 
\begin{theo}\label{classif}
Soit $k$ un corps. Toute $k$-surface projective, lisse, g\'eom\'etriquement rationnelle,
$k$-minimale, appartient \`a au moins un des types suivants :
  
(i) Surface fibr\'ee en coniques relativement minimale au-dessus d'une conique lisse.

(ii) Surface de del Pezzo de degr\'e $d$ avec $1 \leq d \leq 9$.
\end{theo}

Comme observ\'e par Manin et l'auteur (cf. \cite[Thm. IV.6.8]{kollar}),
 ceci permet de d\'emontrer le r\'esultat suivant, qui pour $k$ un corps fini,
admet une d\'emonstration uniforme (A. Weil, cf. \cite[Thm. 27.1, Cor. 27.1.1]{M}).

\begin{prop}\label{C1}
Soient $k$ un corps et $X$ une $k$-surface projective et lisse g\'eom\'etriquement rationnelle.
Si $k$ est un corps $C_{1}$, alors $X$ poss\`ede un point $k$-rationnel.
 Ceci vaut en particulier pour $k$ un corps fini et pour $k=\C((t))$.
 \end{prop}

\section{Fibr\'es en coniques}

Soient $k$ un corps, $\ovk$ une cl\^{o}ture s\'eparable de $k$ et $g={\rm Gal}(\k/k)$.
Si $X$ est une $k$-surface projective lisse g\'eom\'etriquement connexe
munie d'un $k$-morphisme $f :X \to \P^1_{k}$  relativement minimal
dont la fibre g\'en\'erique est une courbe lisse de genre z\'ero,
alors les points ferm\'es $M$ dont la fibre $X_{M}/k(M)$ est non lisse
ont leur corps r\'esiduel $k(M)$ s\'eparable sur $k$, et $X_{M}/k(M)$
se d\'ecompose sur une extension quadratique s\'eparable de corps $L(M)/k(M)$
en deux droites $\P^1_{L(M)}$ qui se coupent transversalement en un $k(M)$-point.
On appelle ces points ferm\'es $M \in \P^1_{k}$ les points de mauvaise r\'eduction de la fibration.
On renvoie \`a \cite{Isk79} pour la d\'emonstration de ces faits.
Ils impliquent que sur une cl\^oture s\'eparable $\k$ de $k$, 
il existe une contraction $\ovX \to Y$ au-dessus de $\P^1_{\k}$ telle que
les fibres de $Y \to \P^1_{\k}$ soient toutes isomorphes \`a $\P^1$.
Ceci d\'efinit donc un \'el\'ement de $\Br(\P^1_{\k})=\Br(\k)=0$.
L'\'egalit\'e $ \Br(\P^1_{\k})=0 $ r\'esulte du th\'eor\`eme de Tsen lorsque $\k$ est alg\'ebriquement clos,
et ce m\^eme th\'eor\`eme implique que $ \Br(\P^1_{\k})$ est $p$-primaire pour $\k$ s\'eparablement clos
de caract\'eristique $p>0$. Que l'on ait $ \Br(\P^1_{\k})=0 $ pour $\k$ s\'eparablement clos quelconque
est \'etabli par Grothendieck \cite[Cor. 5.8]{GrBr3}.
La fibre g\'en\'erique de $\overline{f} : \X \to \P^1_{\k}$ admet donc un point rationnel.
Comme $\P^1_{\k}$ est r\'egulier de dimension 1 et $\overline{f}$ est un morphisme propre,
tout tel point rationnel s'\'etend en une section de $\overline{f} : \X \to \P^1_{\k}$.
La $k$-vari\'et\'e $X$
est d\'eploy\'ee sur $\ovk$.
La fibre g\'en\'erique de la fibration $f : X \to \P^1_{k}$ correspond  \`a un \'el\'ement $\beta$ de
$H^2(g, \k(\P^1)^*)$. En un point ferm\'e $M$ de mauvaise r\'eduction, la fl\`eche diviseur d\'efinit
 une application $g$-\'equivariante
$$ \k(\P^1)^*  \to \oplus_{k(M) \subset \k}  \ \Z,$$
o\`u $k(M) \subset \k$ parcourt les $k$-plongements de l'extension s\'eparable $k(M)/k$
dans $\k$.
Cet homomorphisme induit une fl\`eche r\'esidu
$$\partial_{M} : H^2(g, \k(\P^1)^*) \to H^2(g,  \oplus_{k(M) \subset \k} \  \Z) = H^2(k(M), \Z) = H^1(k(M), \Q/\Z).$$
L'image de $\beta$ par cette application d\'ecrit l'extension quadratique s\'eparable de $k(M)$
correspondant \`a la mauvaise fibre. Pour $k \subset L \subset \k$, avec $L/k$ finie,
on a le diagramme commutatif suivant :
$$\begin{array}{ccccccccc}
 H^2(g_{k}, \k(\P^1)^*) & \to &  H^1(k(M), \Q/\Z)\\
 \downarrow&&\downarrow& \\
 H^2(g_{L}, \k(\P^1)^*) & \to  &\oplus_{N \to M} H^1(k(N),\Q/\Z),
\end{array}$$ 
 o\`u $N$ parcourt les points ferm\'es de $\P^1_{L}$ d'image $M$ via la projection $\P^1_{L} \to \P^1_{k}$,
 o\`u les fl\`eches horizontales sont les fl\`eches de r\'esidu  d\'efinies ci-dessus et les fl\`eches
verticales sont les fl\`eches de restriction.

\begin{lem}\label{utile} 
Soit  $f: X \to \P^1_{k}$ un fibr\'e en coniques relativement minimal 
sur un corps $k$, \`a fibre g\'en\'erique lisse et espace total lisse sur $k$.
Soit   $M \in \P^1_{k}$
un point ferm\'e  de corps r\'esiduel $k(M)$
o\`u la fibration a mauvaise r\'eduction.
Soit $K=k(M)$ le corps r\'esiduel en $M$.
Soit  $N$ un  $K$-point of $\P^1_{K}$ au-dessus de  $M \in \P^1_{k}$.
La fibration  $f_{K}: X_{K}\to \P^1_{K}$ a mauvaise r\'eduction en $N$.
\end{lem}

\begin{proof}
Soit  $A/k(\P^1)$ l'alg\`ebre de quaternions associ\'ee \`a la fibre g\'en\'erique de  $f : X \to \P^1_{k}$.
La fibration $f$  a mauvaise r\'eduction en  $M$ si et seulement si le r\'esidu  $\gamma:=\partial_{M}(A) \in H^1(k(M),\Z/2)=H^1(K,\Z/2)$
est non trivial. D'apr\`es la compatibilit\'e ci-dessus,
on a
  $$\partial_{N}(A_{K})=\gamma \in H^1(K(N),\Z/2)=H^1(K,\Z/2).$$
ce qui \'etablit le lemme.
\end{proof} 

L'\'enonc\'e suivant fut \'etabli par Iskovskikh (\cite[Thm. 2]{Isk70}, \cite[Thm. 4, Thm. 5]{Isk79}).

 \begin{prop}\label{auplus3}
 Soient $k$ un corps et
 $X/k$ une surface projective, lisse, g\'eom\'etriquement connexe
sur $k$, munie d'une structure de fibr\'e en coniques relativement minimale
  $X \to \P^1_{k}$. Si  le nombre $r$ de fibres g\'eom\'etriques d\'eg\'en\'er\'ees
est au plus \'egal \`a 3, et si $X$ poss\`ede un point $k$-rationnel, alors $X$ est une surface $k$-rationnelle.
 \end{prop}

\begin{prop}\label{fibreconiques} 
Soient $k$ un corps et
 $X/k$ une surface projective, lisse, g\'eom\'etriquement connexe
sur $k$, munie d'une structure de fibr\'e en coniques relativement minimale
  $X \to \P^1_{k}$. Supposons $X$  d\'eploy\'ee sur une extension 
  cyclique de corps $K/k$. Si le nombre $r$ de fibres g\'eom\'etriques d\'eg\'en\'er\'ees de
  la fibration est au moins \'egal \`a 4, 
  alors il existe
  une extension finie s\'eparable $k'/k$ telle que $H^1(k',\Pic(\ovX)) \neq 0$.
\end{prop}

\begin{proof}
 Donnons ici quelques rappels sur le module de Picard
d'une surface fibr\'ee en coniques sur $\P^1_{k}$. Pour
plus de d\'etails, je renvoie le lecteur \`a \cite[\S 2]{CTSDMJ81}.

Soit comme ci-dessus $\k$ une cl\^{o}ture s\'eparable de $k$ et $\ovX=X \times_{k}\k$.
On a une suite exacte de modules galoisiens
$$ 0 \to P \to \Z.f \oplus Q \to \Pic(\X) \to \Z \to 0,$$
o\`u  $P$ est le module de permutation sur les  $\k$-points of $\P^1$
\`a fibre singuli\`ere, $Q$ est le module de permutation sur  
les composantes des fibres singuli\`eres sur $\k$, et $\Z.f$ est engendr\'e par une fibre
au-dessus d'un $k$-point de $\P^1_{k}$. L'application $\Pic(\X) \to \Z $ 
est induite par la restriction \`a la fibre g\'en\'erique.
 
Soit  $M$ le noyau de cette application restriction. On a des suites exactes
courtes de modules galoisiens
$$ 0 \to P \to \Z \oplus Q \to M \to 0$$
et
$$0 \to M \to \Pic(\X) \to \Z \to 0.$$
Par cohomologie galoisienne on obtient des suites exactes
$$ 0 \to \Z/2 \to H^1(k,M) \to H^1(k, \Pic(\X)) \to 0$$
et
$$0 \to H^1(k,M) \to H^2(k,P) \to H^2(k,  \Z \oplus Q).$$
Cette derni\`ere donne naissance \`a une suite exacte
$$0 \to H^1(k,M) \to \oplus_{i=1}^r \Z/2 \to H^1(k,\Z/2),$$
o\`u  $i$ parcourt les $r\geq 1$   points ferm\'es $P_{i}$ de $\P^1$ \`a fibre singuli\`ere, 
d\'eploy\'ee par une extension quadratique s\'eparable  de corps,  de classe $a_{i} \in H^1(k(P_{i}),\Z/2)$,
et l'application $\theta_{i} : \Z/2 \to H^1(k,\Z/2)$ envoie $1$
sur la norme (de $k(P_{i})$  \`a $k$) de   $a_{i} \in H^1(k(P_{i}),\Z/2)$.  

 On a en outre une relation de r\'eciprocit\'e \cite[\S 2, Remark]{CTSDMJ81} 
 qui implique ici que l'image de l'\'el\'ement $(1,\dots,1) \in \oplus_{i} H^1(k(P_{i}),\Z/2)$ 
 est la classe triviale dans $H^1(k,\Z/2)$.

Nous voulons montrer : si le nombre de fibres g\'eom\'etriques d\'eg\'en\'er\'ees de $X \to \P^1_{k}$
est sup\'erieur ou \'egal \`a 4, alors 
 il existe  une extension finie s\'eparable $k'/k$ telle que $H^1(k',\Pic(\ovX)) \neq 0$.

Si $E/k$ est une extension s\'eparable de corps de degr\'e impair, 
par passage \`a $E$, la famille $X_{E} \to \P^1_{E}$ reste relativement minimale :
aucun r\'esidu n'est annul\'e. On peut donc supposer 
que tous les points ferm\'es de mauvaise r\'eduction sont soit de degr\'e 1 soit de degr\'e pair.

Supposons qu'il existe un point ferm\'e de mauvaise r\'eduction $P$ de degr\'e pair au moins \'egal \`a 4.
Par le lemme \ref{utile}, on se ram\`ene apr\`es extension de $k$ \`a $k(P)$
\`a la situation o\`u il y a au moins 4 points rationnels $P_{1},P_{2},P_{3},P_{4}$
\`a fibre singuli\`ere. Comme la surface est par hypoth\`ese d\'eploy\'ee par une extension cyclique,
les classes $a_{i} \in H^1(k(P_{i}),\Z/2)=H^1(k,\Z/2)$, qui sont non triviales, co\"{\i}ncident toutes
avec une m\^eme classe non triviale $a \in H^1(k,\Z/2)$.

L'application $\oplus_{i=1}^r \Z/2 \to H^1(k,\Z/2)$ induit une
application $(\Z/2)^4 \to H^1(k, \Z/2)$  
qui se factorise donc par $(\Z/2)^4 \to \Z/2$.
Le groupe $H^1(k,M)= Ker [\oplus_{i=1}^r \Z/2 \to H^1(k,\Z/2)]$
  contient donc au moins   le noyau d'une
application   $(\Z/2)^4 \to \Z/2$, il est donc d'ordre au moins 8,
et $H^1(k, \Pic(\ovX))$ est donc d'ordre au moins 4.

Supposons d\'esormais que les points ferm\'es $P_{i}$ de mauvaise r\'eduction
sont tous de degr\'e 2 ou 1. 

S'il y a au moins 4 points ferm\'es $P_{i}$ de mauvaise r\'eduction
de degr\'e 1 sur $k$, l'argument ci-dessus permet de conclure.

Supposons qu'il y a au moins deux points ferm\'es $P_{1}, P_{2}$ de degr\'e 2. 
Comme la surface est par hypoth\`ese 
d\'eploy\'ee sur une  extension cyclique, ceci impose que les extensions
$k(P_{i})/k$ co\"{\i}ncident en ces deux points. 
Soit donc $L/k$ l'extension quadratique  s\'eparable de corps ainsi d\'efinie.
En appliquant le lemme \ref{utile}, on se ram\`ene apr\`es extension de $k$ \`a $L$ au cas 
o\`u la fibration $X \to \P^1_{k}$ a des fibres singuli\`eres  au-dessus d'au moins 4
  points $k$-rationnels, cas qui a d\'ej\`a \'et\'e r\'egl\'e.

On peut donc supposer que l'ensemble des degr\'es des points ferm\'es \`a
mauvaise r\'eduction est soit $(2,1,1,1)$ soit $(2,1,1)$.
  
  Consid\'erons le cas $(2,1,1,1)$.
 Comme la surface est par hypoth\`ese 
d\'eploy\'ee sur une  extension cyclique, ceci impose que les extensions quadratiques
associ\'ees \`a
$a_{i} \in H^1(k(P_{i}), \Z/2)$ pour $P_{i}$ $k$-rationnel  co\"{\i}ncident
avec une m\^eme classe $a \in H^1(k,\Z/2)$. Soit $R$ le point ferm\'e de degr\'e 2,
et soit $b \in H^1(k(R),\Z/2)$ le r\'esidu en ce point.
L'application
$\oplus_{i=1}^3 \Z/2  \oplus \Z/2  \to H^1(k,\Z/2)$
envoie $(x,y,z,t)$ sur $(x+y+z).a+ t.Norm_{k(R)/k}(b)\in H^1(k,\Z/2)$.
On sait que la classe $(1,1,1,1)$ a une image triviale.
Donc $3a+Norm_{k(R)/k}(b)$ est trivial dans $H^1(k,\Z/2)$.
Ceci implique $a=Norm_{k(R)/k}(b)$, et cet \'el\'ement est non trivial dans $H^1(k,\Z/2)$.
L'application $\oplus_{i=1}^3 \Z/2  \oplus \Z/2  \to H^1(k,\Z/2)$
envoie donc $(x,y,z,t)$ sur $(x+y+z+t).a $ dans $H^1(k,\Z/2)$.
C'est donc simplement la somme
$(\Z/2)^4 \to \Z/2$. Son noyau est $(\Z/2)^3$, on a donc
$H^1(k, \Pic(\ovX)) =(\Z/2)^2$.

Montrons pour finir  que le cas $(2,1,1)$ n'existe pas.
Comme ci-dessus, les
extensions quadratiques
associ\'ees \`a
$a_{i} \in H^1(k(P_{i}), \Z/2)$ pour $P_{i}$ $k$-rationnel  co\"{\i}ncident
avec une m\^eme classe non triviale $a \in H^1(k,\Z/2)$. Soit $Q$ le point ferm\'e de degr\'e 2,
et soit $b \in H^1(k(Q),\Z/2)$ le r\'esidu en ce point. Ceci correspond \`a une
extension quadratique s\'eparable $L/k(Q)$. Sous nos 
hypoth\`eses, l'extension $L/k$ est cyclique de groupe de Galois $\Z/4$. Sous cette 
hypoth\`ese, on v\'erifie
 que la norme $H^1(k(Q), \Z/2) \to H^1(k,\Z/2)$
envoie la classe $b$ sur la classe $a$.
Par ailleurs on sait (r\'eciprocit\'e) que la classe $(1,1,1)$ a une image triviale
par l'application  
$\oplus_{i=1}^2 \Z/2  \oplus \Z/2  \to H^1(k,\Z/2)$.
Mais ceci dit que $3a =a \in H^1(k,\Z/2)$ est nul. Contradiction.
 \end{proof}

\section{Surfaces de del Pezzo}

On a l'\'enonc\'e connu (Ch\^{a}telet, Manin, Iskovskikh, voir \cite[Thm. 2.1]{VA}):
 
\begin{prop}\label{granddp}
Soit $X$ une surface de del Pezzo de degr\'e $d \geq 5$ sur un corps $k$.
Si $X$ poss\`ede un point $k$-rationnel, alors $X$ est $k$-rationnelle.
\end{prop}

\begin{prop}\label{dp4}
Soient $k$ un corps et  $X \subset \P^4_{k}$ une surface de del Pezzo de degr\'e $4$ sur $k$ 
 d\'eploy\'ee par une extension cyclique $K/k$.
Supposons que $X$ est $k$-minimale. Alors :

 (i) Il existe une extension de corps $E/k$ telle que $H^1(E, \Pic(\X)) \neq 0$.
 
 (ii) Le module galoisien $\Pic(\X)$ n'est pas facteur direct d'un module de permutation.
 
 (iii) Si $X$ poss\`ede un $k$-point, 
il existe  une extension finie s\'eparable
 $k'/k$ telle que $H^1(k', \Pic(\X)) \neq 0$.

  \end{prop}

  \begin{proof}
  Soit $X$ une surface de del Pezzo de degr\'e 4 sur un corps $k$, poss\'edant un $k$-point,
  d\'eploy\'ee sur une extension cyclique de $k$.
 Supposons $X$ $k$-minimale.  D'apr\`es \cite[Thm. 2]{Isk72}, la surface $X$ n'est pas $k$-rationnelle.
 
 Si $X$ poss\`ede un $k$-point $P$ non situ\'e sur les 16 droites
 (g\'eom\'etriques) exceptionnelles, en \'eclatant le point $P$ on obtient une surface cubique lisse $Y$ sur $k$
 \'equip\'ee d'une fibration en coniques sur $Y \to \P^1_{k}$, avec 5 fibres g\'eom\'etriques d\'eg\'en\'er\'ees,
 d\'eploy\'ee sur une extension cyclique de $k$.  Si cette fibration n'est pas relativement minimale, soit
 $Z \to \P^1_{k}$  un mod\`ele relativement minimal. Soit $r$ le nombre de fibres g\'eom\'etriques
 d\'eg\'en\'er\'ees. Si $r \leq 3$, alors d'apr\`es la Proposition \ref{auplus3}, $Y$ est $k$-rationnelle,
 ce qui est exclu. Ainsi on a $r\geq 4$. 
 La Proposition \ref{fibreconiques}  donne alors l'existence
 d'une extension finie s\'eparable $k'/k$ telle que  $H^1(k', \Pic(\overline{Z})) \neq 0$. 
 Comme on passe de $X$ \`a $Z$ par des s\'eries d'\'eclatements en des points ferm\'es
 s\'eparables, 
 les modules galoisiens $\Pic(\overline{X})$ et $\Pic(\overline{Z})$
sont isomorphes \`a addition de modules de permutation pr\`es.
On a donc  $H^1(k', \Pic(\overline{X})) \neq 0$.

Si $X$ poss\`ede un $k$-point, et le corps $k$ poss\`ede au moins 23 \'el\'ements, alors il existe
 un $k$-point sur $X$ hors des 16 droites \cite[Chap. 4, \S 8, Teor. 8.1= Thm. 30.1]{M}. 
 Supposons que $k$ est fini et $X$ est d\'eploy\'ee sur l'extension cyclique $K/k$.
 Il existe une extension finie $L/k$ lin\'eairement disjointe de $K$ sur laquelle $X$ poss\`ede un 
 point $L$-rationnel non situ\'e hors des 16 droites. La $L$-surface de del Pezzo de degr\'e 4 $X\times_{k}L$
 est $L$-minimale et d\'eploy\'ee par l'extension cyclique $K.L/L$. L'argument ci-dessus donne alors une
 extension finie $k'$ de $L$ telle que $H^1(k', \Pic(\overline{X})) \neq 0$.
 
 Ceci \'etablit le point (iii). 
 
  Pour \'etablir (i), on utilise l'astuce du passage au point g\'en\'erique (voir 
 \cite[Thm. 2.B.1]{CTSDesc}). Soit $F=k(X)$ le corps des fonctions de $X$. 
 La $F$-vari\'et\'e $X_{F}=X\times_{k}F$ poss\`ede un $F$-point.
 Elle est $F$-minimale, car $k$ est alg\'ebriquement clos dans $F=k(X)$.
 Le module galoisien $\Pic(\overline{X})$ ne change pas par passage du corps de base de $k$ \`a $F$,
 il est d\'eploy\'e par l'extension  $k'/k$ comme par l'extension $F'/F$, o\`u $F':=F.K$.
  Comme le corps $F$
 est infini, il existe un $F$-point de $X$ non situ\'e sur les droites (g\'eom\'etriques) de $X$.
 La $F$-surface  minimale  $X_{F}$ est    d\'eploy\'ee par l'extension cyclique $F'/F$.
 Par le point (iii), il existe une extension finie s\'eparable $E/F$ telle que $H^1(E, \Pic(\overline{X})) \neq 0$,
 ce qui donne (i).
 
 Ceci implique que le module galoisien $\Pic(\overline{X}) $ n'est pas facteur direct d'un module de permutation
(Th\'eor\`eme \ref{implgen}), ce qui donne (ii).
 \end{proof}

   On s'int\'eresse maintenant aux  surfaces de del Pezzo de degr\'e 3, 2 et 1 d\'eploy\'ees par une extension cyclique $K/k$
   du corps de base $k$. 
   On note $Frob$ un g\'en\'erateur  du groupe cyclique $\Gal(K/k)$.
Les diverses actions du groupe cyclique $\Gal(K/k)$
sur le groupe $\Pic(X_{K})=\Pic(\ovX)$ ont \'et\'e classifi\'ees par
Frame \cite{Frame},  puis Swinnerton-Dyer \cite{SwD},  corrig\'ees par  Manin \cite[Chapitre IV]{M},   corrig\'ees
et compl\'et\'ees par Urabe \cite{U}  puis r\'ecemment par Banwait, Fit\'e et Loughran \cite{BFL}.

   Dans   \cite[Chap. IV, Table I,  Colonne 5]{M} et dans \cite[Table 7.1, Colonne 5]{BFL}, une surface a
le symbole  $\prod_{m}m^{n_{m}}$, avec tous les $n_{m}\geq 0$,
si  pour $m$ donn\'e 
l'ensemble Galois invariant des
racines primitives $m$-i\`emes de l'unit\'e parmi les valeurs propres de $Frob$
a $n_{m}$ \'el\'ements. En d'autres termes, on d\'ecompose le polyn\^ome caract\'eristique
de $Frob$ agissant sur   $\Pic(\ovX)\otimes_{\Z}{\bf C}$
en regroupant les orbites des racines sous l'action du groupe de Galois.
On appellera ce  symbole  le  symbole  caract\'eristique de $Frob$ (pour son action sur $\Pic(\ovX)\otimes_{\Z}{\bf C}$).

Urabe \cite[Supplement]{U} utilise le symbole de Frame \cite{Frame}.
Le symbole de Frame $\prod_{m}m^{n_{m} }$,  avec $n_{m} \in \Z$,
correspond \`a une r\'e\'ecriture du polyn\^{o}me caract\'eristique de $Frob$
pour son action sur $\Pic(\ovX)\otimes_{\Z}{\bf C}$
comme un produit  $\prod_{m}(t^m-1)^{n_{m}}$,  avec $n_{m} \in \Z$.
 Il y a une unique fa\c con d'\'ecrire  le polyn\^{o}me caract\'eristique comme un tel produit
 (avec des entiers $m>0$ distincts non nuls, et des $n_{m}$ non nuls).
  Soit $r>1$. Pour calculer le symbole de Frame de $Frob^{r}$, 
dans le produit  $\prod_{m}(t^m-1)^{n_{m}}$ attach\'e au symbole de Frame  $\prod_{m}m^{n_{m} }$ de $Frob$,
  pour chaque entier $m$ on \'ecrit 
  $r=uv$ et   $m=uw$ ($u, v, w$ entiers positifs) avec $(v,w)=1$,  et
 on remplace   $(t^m-1)$ 
 par   $(t^w-1)^u$, puis on regroupe les termes.

Dans les tables de \cite{U} et \cite{BFL}, les symboles (d'un type ou de l'autre type)  
associ\'es  \`a des   surfaces diff\'erentes peuvent co\"{\i}ncider
mais c'est rare.

\bigskip
  
 La proposition suivante m'a \'et\'e indiqu\'ee par K. Shramov.
 \begin{prop} (A. Trepalin) \label{dp3}
  Soit  $X$ une surface cubique lisse sur un corps $k$, 
  d\'eploy\'ee par une extension cyclique $K/k$.
Supposons que $X$ est $k$-minimale.
Il existe alors une extension finie s\'eparable de corps
  $k'/k$ telle que $H^1(k', \Pic(\X)) \neq 0.$
\end{prop}

\begin{proof}

  Nous utilisons ici la table 7.1 de l'article 
 \cite{BFL}.    
Les actions correspondant \`a des surfaces $k$-minimales,
 c'est-\`a-dire  d'indice $i(X)=0$, sont celles num\'erot\'ees 1, 2, 3, 4, 5 dans la table 7.1
 de \cite{BFL}.
 Pour les num\'eros 3 et 5, on a donc $H^1(k, \Pic(\overline{X})) \neq 0$.
 Pour les autres, on a $H^1(k,\Pic(\overline{X})) = 0$.
 
 Pour le num\'ero 1, les valeurs propres de  $Frob$  sont $1, 3^2, 12^4$,
 c'est-\`a-dire $1$, les deux racines cubiques primitives de l'unit\'e et les 4 racines primitives
 12-i\`emes de l'unit\'e. Si on remplace $Frob$ par $Frob^4$,  c'est-\`a-dire si on passe 
 \`a l'extension $k'/k$ de degr\'e 4, les valeurs propres de $Frob^4$ sont $1, 3^6$. Dans la table, seul
 le num\'ero 3 a ces valeurs propres. Et \`a ce niveau on a $H^1(k', \Pic(\overline{X})) \neq 0$.
 
 Pour le num\'ero 2, les valeurs propres de $Frob$ sont $1,3^2,6^4$. 
 Si on remplace $Frob$ par $Frob^2$,  c'est-\`a-dire si on passe  \`a l'extension $k'/k$ de degr\'e 2,
 les valeurs propres de $Frob^2$  sont
  $1, 3^6$. Dans la table, seul
 le num\'ero 3 a ces valeurs propres. Et \`a ce niveau on a $H^1(k', \Pic(\overline{X})) \neq 0$.
 
 Pour le num\'ero 4,  les valeurs propres de $Frob$  sont $1, 9^6$. 
 Si on remplace $Frob$ par $Frob^3$, c'est-\`a-dire si on passe  \`a l'extension $k'/k$ de degr\'e 3,
 les valeurs propres de $Frob^3$  sont
  $1, 3^6$. Dans la table, seul le num\'ero 3 a ces valeurs propres. Et \`a ce niveau on a $H^1(k', \Pic(\overline{X})) \neq 0$.
\end{proof}

 \begin{prop}\label{dp2}
  Soit  $X$ une surface de del Pezzo de degr\'e 2 sur un corps~$k$, 
  d\'eploy\'ee par une extension cyclique $K/k$.
  Supposons que $X$ est $k$-minimale.
Il existe alors une extension finie s\'eparable de corps
  $k'/k$ telle que $H^1(k', \Pic(\X)) \neq 0.$
 \end{prop}
 \begin{proof}
 On utilise ici la table 1 de l'article \cite{U} de T. Urabe
 et   les symboles de Frame.
 
Il suffit de discuter les surfaces num\'erot\'ees de  1 \`a 19, qui correspondent \`a des
surfaces $k$-minimales.
Le cas 1 de la table 1, comme Daniel Loughran me l'a signal\'e, contient une erreur.
Son indice n'est pas 0, il est au moins 2, la surface n'est pas $k$-minimale.
On ne discute donc que les cas 2 \`a 19.

Pour les surfaces avec $H^1(k, \Pic(\ovX))  \neq 0$ il n'y a rien \`a faire.

 Dans chacun des cas ci-dessous, on consid\`ere une puissance $Frob^r$
 de $Frob$ et on note $k'$ le corps fixe de $Frob^r$.

 Cas 5. En prenant $Frob^5$, on trouve
     $1^{-4}.2^6$ comme nouveau symbole de Frame.
     La seule possibilit\'e est le cas~2, si $k'$ est le corps fixe de $Frob^5$,
      on  a  $H^1(k',\Pic(\overline{X})) \neq 0$.
 
 Cas 6.  En prenant $Frob^2$, on trouve
   $4^2$. La seule possibilit\'e est le cas~3, on a $H^1(k',\Pic(\overline{X})) \neq 0$.
 
 Cas 7. En prenant $Frob^3$, on trouve $1^{-4}.2^6$.
  La seule possibilit\'e est le cas~2, on a $H^1(k',\Pic(\overline{X})) \neq 0$.
  
 Cas 15. En prenant $Frob^9$, on trouve  $1^{-4}.2^6$.
 La seule possibilit\'e est le cas~2, on a $H^1(k',\Pic(\overline{X})) \neq 0$.

 Cas 16. En prenant $Frob^7$, on trouve  $1^{-6}. 2^7$.
 La seule possibilit\'e est le cas~8, on a $H^1(k',\Pic(\overline{X}) ) \neq 0$.
 
 Cas 17. En prenant $Frob^3$, on trouve $1^{-2}.2^1.4^2$. 
 La seule possibilit\'e est le cas~9, on a $H^1(k',\Pic(\overline{X})) \neq 0$.
 
 Cas 18. En prenant $Frob^3$, on trouve $1^{-1}.2^2.5^{-1}.10^1$.
 La seule possibilit\'e est le cas~13, on a $H^1(k',\Pic(\overline{X})) \neq 0$.
 
 Case 19.  En prenant $Frob^3$, on trouve  $1^{-6}.2^7$. 
La seule possibilit\'e est le cas~8, on a $H^1(k',\Pic(\overline{X})) \neq 0$.
  \end{proof}

   \begin{rema}
  Dans \cite{L}, l'auteur mentionne trois types de  surfaces de del Pezzo  de degr\'e 2
  dans la table 1 d'Urabe qui auraient tous leurs groupes $H^1(k',\Pic(\overline{X}))$  triviaux. Ce sont les
  types 1, 5 et 16. Pour le cas 1, nous avons  vu que ce n'est pas une surface $k$-minimale.
  Pour les deux autres cas, il doit s'agir d'une erreur de calcul dans \cite{L}.
   \end{rema}
 
  \begin{prop}\label{dp1}
 Soit  $X$ une surface de del Pezzo de degr\'e 1 sur un corps~$k$, 
  d\'eploy\'ee par une extension cyclique $K/k$.
 Supposons que $X$ est $k$-minimale.
Il existe alors une extension finie s\'eparable
  $k'/k$ telle que $H^1(k', \Pic(\X)) \neq 0.$
 \end{prop}
 \begin{proof}
  On utilise ici l'article \cite{U} de Urabe et sa   table~2. 
  On ne consid\`ere que les surfaces d'indice  0. 
  Si  $H^1(k,\Pic(\overline{X})) \neq 0$, on a fini.
 Sinon, on est dans l'un des cas suivants, 
 pour lesquels on   utilise les symboles de Frame.
 On note $Frob$ un g\'en\'erateur de $\Gal(K/k)$.
 Dans chacun des cas ci-dessous, on consid\`ere une puissance $Frob^r$
 de $Frob$ et on note $k'$ le corps fixe de $Frob^r$.
   
 Cas 5.  En prenant $Frob^3$,  on trouve $1^1.2^{-2}.4^3$ comme nouveau symbole de Frame.
  La seule possibilit\'e est le cas 3,  on a $H^1(k',\Pic(\overline{X})) \neq 0$.
 
 Cas 6.  En prenant $Frob^5$,  on trouve
    $1^{-3}.2^4.4^1$. La seule possibilit\'e est le cas 1, on a
   $H^1(k',\Pic(\overline{X})) \neq 0$.
   
 Case 7. En prenant $Frob^3$,  on trouve
  $1^{-1}.2^3.4^{-1}.8^1$. La seule possibilit\'e est le cas 4, on a $H^1(k',\Pic(\overline{X})) \neq 0$.

 Cas 29. En prenant $Frob^{10}$,  on trouve
 $1^{-3}.3^4$. 
 La seule possibilit\'e est le cas  9,  on a $H^1(k',\Pic(\overline{X})) \neq 0$.
  
 Cas 30. En prenant $Frob^3$,  on trouve
 $1^1.4^{-2}.8^2$. 
La seule possibilit\'e est le cas 23, on a $H^1(k',\Pic(\overline{X})) \neq 0$.
 
  Cas 31. En prenant $Frob^5$,  on trouve
  $1^1.2^{-4}.4^4$. 
 La seule possibilit\'e est le cas
  17, on a $H^1(k',\Pic(\overline{X})) \neq 0$.
   
 Cas 32. En prenant $Frob^3$,  
  on trouve $1^1.2^{-4}.4^4$. 
 La seule possibilit\'e est le cas  17, on a $H^1(k',\Pic(\overline{X})) \neq 0$.

 Cas 33. En prenant $Frob^2$, on trouve
  $9^1$. 
 La seule possibilit\'e est le cas    14,  on a $H^1(k',\Pic(\overline{X})) \neq 0$.
 
 Cas 34. En prenant $Frob^3$,  on trouve
 $1^{-1}.5^2$. 
 La seule possibilit\'e est le cas    11, on a $H^1(k',\Pic(\overline{X})) \neq 0$.
 
 Cas 35. En prenant $Frob^2$, on trouve $1^{-1}.5^2$. 
La seule possibilit\'e est le cas  11, $H^1(k',\Pic(\overline{X})) \neq 0$.
 
 Cas 36. En prenant $Frob^3$,  on trouve
    $1^{-3}.2^2.4^2$. 
 La seule possibilit\'e est le cas 10, $H^1(k',\Pic(\overline{X})) \neq 0$.

 Cas 37. En prenant $Frob^2$, on trouve 
  $1^{-3}.3^4$. 
 La seule possibilit\'e est le cas 9, on a $H^1(k',\Pic(\overline{X})) \neq 0$.
   \end{proof}
   
   \begin{rema} Comme le note un rapporteur, dans les propositions  \ref{dp1} et \ref{dp2} un certain nombre   de cas
   rel\`event  aussi du cas des surfaces fibr\'ees en coniques, d\'ej\`a trait\'ees dans la proposition
   \ref{fibreconiques}. 
      \end{rema}
   
   \begin{rema}
   L'\'enonc\'e du th\'eor\`eme 5.4.3  de \cite{L} est identique \`a l'\'enonc\'e ci-dessus.
   Il convient n\'eanmoins de corriger la d\'emonstration donn\'ee dans \cite{L} pour le cas 6,
   car le choix fait l\`a de $Frob^4$ m\`ene, comme le dit l'auteur \`a la surface
   num\'ero 102, mais cette derni\`ere a $H^1=0$.
       \end{rema}
  
  \bigskip 
 
 \section{Conclusion}
 
 R\'ecapitulons. Il s'agit de montrer :
 
 \begin{theo}\label{vraiprincipalbis}
Soient $k$ un corps et $X$ une $k$-surface projective, lisse, g\'eom\'etriquement rationnelle.
Supposons que $X$ poss\`ede un point $k$-rationnel et que
$X$ soit d\'eploy\'ee par une extension cyclique de $k$.
Si $X$ n'est pas  $k$-rationnelle, alors :

(i) Il existe une extension $k'/k$ finie s\'eparable telle que
 $H^1(k', \Pic(\X))\neq 0$.
  
 (ii) 
Le module galoisien $\Pic(\X)$ n'est pas un facteur direct d'un module de
permutation.

(iii) 
La $k$-vari\'et\'e $X$ n'est pas stablement $k$-rationnelle.
  \end{theo}

  \begin{proof} 
  On peut supposer que $X$ est $k$-minimale, car si $f : Y \to X$ est un $k$-morphisme birationnel de $k$-surfaces projectives lisses
  g\'eom\'etriquement rationnelles, si $Y$ est d\'eploy\'ee par une extension cyclique de $k$, il en est de m\^{e}me de $X$.
  D'apr\`es le th\'eor\`eme \ref{classif}, on peut en outre supposer  que $X$   ou bien est une surface de del Pezzo 
  $k$-minimale de degr\'e $d$ avec $1 \leq d \leq 9$, ou bien est munie d'une fibration en coniques relativement minimale au-dessus de $\P^1_{k}$.

  Si $X$ est une surface de del Pezzo de degr\'e $d$ avec $5 \leq d \leq 9$, alors
  $X$ est $k$-rationnelle d'apr\`es les propositions \ref{C1} et \ref{granddp}.

  Si $X$ est une surface de del Pezzo $k$-minimale de degr\'e $d=4$, resp. $d=3$, resp.
  $d=2$, resp. $d=1$, alors d'apr\`es la proposition  \ref{dp4}   , resp. \ref{dp3},   resp.  \ref{dp2}, resp. \ref{dp1}, 
  il existe une extension finie
 s\'eparable  $k'/k$ avec $H^1(k', \Pic(\X))\neq 0$.

  Si $X$ est munie d'une fibration en coniques $X \to \P^1_{k}$ relativement $k$-minimale,
  les propositions \ref{auplus3} et \ref{fibreconiques}   assurent que soit $X$ est $k$-rationnelle,  soit il existe une extension finie
 s\'eparable  $k'/k$ avec $H^1(k', \Pic(\X))\neq 0$.  
 
Ceci donne (i), et les autres \'enonc\'es suivent (Th\'eor\`eme \ref{implgen}).
 \end{proof}

\begin{rema}

On aimerait avoir une d\'emonstration  du th\'eor\`eme \ref{vraiprincipalbis}
qui ne passe pas par l'analyse cas par cas utilis\'ee dans le pr\'esent article, et sp\'ecialement
qui \'evite celle utilis\'ee pour les surfaces de del Pezzo de degr\'e $1$ \`a $3$.
Pour les termes employ\'es dans ce qui suit, on renvoie \`a \cite{CTSDesc}.
Soit $k$ un corps quasi-fini,  et 
soit $X$ une  $k$-surface projective, lisse, g\'eom\'etriquement rationnelle,
poss\'edant un $k$-point.
Comme le groupe de Galois absolu de $k$ est procyclique,
l'hypoth\`ese  $H^1(k', \Pic(\X))=0$ pour toute extension finie $k'/k$,
implique que le module galoisien $\Pic(\X)$ est un facteur direct d'un module de permutation
(th\'eor\`eme d'Endo et Miyata, cf. \cite[Prop. 2, p. 184]{CTSRequiv}). Le corps $k$ est de dimension cohomologique 1.
Soit $S$ le $k$-tore dual du module galoisien $\Pic(\X)$. Il existe alors un unique
torseur universel ${\mathcal T} \to X$ sur $X$ (\`a isomorphisme pr\`es).
C'est un torseur sous le $k$-tore $S$, lequel est un facteur direct d'un $k$-tore
quasi-trivial. Ce torseur est donc g\'en\'eriquement scind\'e (th\'eor\`eme 90 de Hilbert).
La $k$-vari\'et\'e ${\mathcal T}$ est donc $k$-birationnelle au produit
$X \times_{k} S$. C'est une question ouverte de savoir si l'espace total d'un torseur universel
 ${\mathcal T}$  avec un $k$-point au-dessus d'une surface g\'eom\'etriquement rationnelle $X$ est une vari\'et\'e (stablement) $k$-rationnelle. Si c'\'etait le cas, 
l'argument ci-dessus montrerait au moins que, sous l'hypoth\`ese
 $H^1(k', \Pic(\X))=0$ pour toute extension finie $k'/k$,  la surface $X$
 est facteur direct d'une $k$-vari\'et\'e $k$-rationnelle.
\end{rema}  
       
      \bigskip
      
    {\bf Remerciements}.   A. Pirutka m'a  signal\'e la question de B. Hassett.
    Je remercie K. Shramov de m'avoir montr\'e la proposition 
    \ref{dp3} (A. Trepalin) pour les surfaces cubiques.  
 Je remercie D. Loughran de m'avoir donn\'e des pr\'ecisions sur l'article \cite{U}.
 Apr\`es avoir vu une premi\`ere version du pr\'esent article,
il a aussi attir\'e mon attention sur l'article \cite{L}, dont certains des
 calculs pour les surfaces de del Pezzo co\"{\i}ncident avec ceux des propositions \ref{dp2} et \ref{dp1} ci-dessus. Des erreurs dans \cite{L}
 n'ont pas permis \`a l'auteur d'obtenir le r\'esultat g\'en\'eral pour les surfaces de del Pezzo de degr\'e 2.
 Les critiques de  deux rapporteurs sur la version initiale de cet article m'ont permis de pr\'eciser certains points.


\begin{thebibliography}{99}

 
\bibitem{BFL}  B. Banwait, F. Fit\'e et D. Loughran,
 Del Pezzo surfaces over finite fields and their Frobenius traces,
 Mathematical Proceedings of the Cambridge Philosophical Society, \`a para\^{\i}tre.
https://arxiv.org/pdf/1606.00300.pdf

\bibitem{BCTSaSD} A. Beauville, J.-L. Colliot-Th\'el\`ene, J.-J. Sansuc et Sir Peter Swinnerton-Dyer,
Vari\'et\'es stablement rationnelles non rationnelles,
Ann. of Math.  {\bf 121} (1985) 283--318.

\bibitem{CTBarbara}  J.-L. Colliot-Th\'el\`ene, Birational invariants, purity, and the Gersten conjecture,
in  {\it K-Theory and Algebraic Geometry: Connections with Quadratic Forms and Division Algebras}, AMS Summer Research Institute, Santa Barbara 1992, ed. W. Jacob and A. Rosenberg, Proceedings of Symposia in Pure Mathematics {\bf 58}, Part I (1995) 1--64.

\bibitem{CTCoray} J.-L. Colliot-Th\'el\`ene et D. Coray,  L'\'equivalence rationnelle sur les points ferm\'es des surfaces rationnelles fibr\'ees en coniques, Compositio Mathematica {\bf 39} no. 3 (979) 301--332.



\bibitem{CTSRequiv}   J.-L. Colliot-Th\'el\`ene et J.-J. Sansuc, La R-\'equivalence sur les tores,
Ann. sci. \'Ec. Norm. Sup.  (4) {\bf 10}  (1977) 175--229.


\bibitem{CTSDesc}   J.-L. Colliot-Th\'el\`ene et J.-J. Sansuc, La descente sur les vari\'et\'es rationnelles, II,  Duke Math. J. {\bf 54}  (1987) 375--492.

\bibitem{CTSDMJ81} J.-L. Colliot-Th\'el\`ene et  J.-J. Sansuc, On the Chow group of rational surfaces: a sequel to a paper of S.~Bloch, Duke Math. J.  {\bf 48} no. 2  (1981) 421--447.

\bibitem{Frame} J. S. Frame, The classes and representations of the groups of 27 lines and 28 bitangents,
Ann. Math. Pura Appl. (4) {\bf 32} (1951), 83--119.

\bibitem{Fulton} W. Fulton, {\it Intersection Theory}, Ergebnisse der Mathematik und ihrer Grenzgebiete, 3. Folge, Band 2, Springer-Verlag (1984).

\bibitem{Gille} S. Gille, Permutation modules and Chow motives of geometrically rational surfaces,
with appendix by J.-L. Colliot-Th\'el\`ene, J. Algebra {\bf 440} (2015) 443--463.

\bibitem{GrBr3} A. Grothendieck, Le groupe de Brauer, III, in {\it Dix expos\'es sur la cohomologie des sch\'emas}, Masson, North-Holland, Paris, 1968, p. 88--188.

 

\bibitem{Isk70} V. A. Iskovskikh, Surfaces rationnelles avec un pinceau de courbes rationnelles et avec carr\'e de la classe canonique
positif (en russe),
Mat. Sbornik {\bf 83} (125)  no. 1 (1970) 90--119. Trad. ang. Math. USSR-Sb. {\bf 12} (1970), 91--117.

\bibitem{Isk72} V. A. Iskovskikh, Propri\'et\'es birationnelles d'une surface de degr\'e 4 dans ${\bf P}^4_{k}$,
Mat. Sbornik {\bf 88} (130)  no. 1  (1972), trad. ang. Math. USSR Sbornik {\bf 17} no. 1 (1972).

\bibitem{Isk79} V. A. Iskovskikh,  Mod\`eles minimaux des surfaces rationnelles sur un corps
quelconque (en russe),
Izv. Akad. Nauk SSSR Ser. Mat. {\bf 43} (1979), no. 1, 19--43, 237, trad. ang.  Math. USSR Izvestija {\bf 14} no1 (1980) 17--39.

\bibitem{kollar} J. Koll\'ar, Rational Curves on Algebraic Varieties, Ergebnisse der Mathematik und ihrer Grenzgebiete,
3. Folge, Band 3, Springer-Verlag, Berlin Heidelberg (1996).

\bibitem{L} Shuijing Li, Rational points on del Pezzo surfaces of degree 1 and 2,  th\`ese, Rice University,
https://arxiv.org/pdf/0904.3555.pdf

\bibitem{M}  Yu. I. Manin, {\it Formes cubiques}, Nauka, Moscou, 1972 (en russe).
Traduction en anglais (avec une num\'erotation diff\'erente), 2\`eme \'edition, North Holland 1986. 
 
  
\bibitem{P}  A. Pirutka, Varieties that are not stably rational, zero-cycles and unramified cohomology,
in Proc. AMS 2015 Summer Conference (Salt Lake City).

\bibitem{serre}
 J-P. Serre, Corps locaux, Publications de l'Institut de Math\'ematique de l'Universit\'e de Nancago,
 Actualit\'es scientifiques et industrielles {\bf 1296}, Hermann, Paris, 1968.
 
 \bibitem{serre2}  J-P. Serre, Cohomological invariants, Witt invariants and trace forms, in {\it Cohomological invariants in Galois Cohomology}, University Lecture Series {\bf 28} (2003) Amer. Math. Soc.

\bibitem{SwD} H.P.F.  Swinnerton-Dyer,  The zeta-function of a cubic surface over a finite field,
Math. Proc. Camb. Phil. Soc. {\bf 63} (1967) 55--71.
 
\bibitem{U} T. Urabe, Calculations of Manin's invariant for del Pezzo surfaces,
 Mathematics of computations, Volume {\bf 65}, Number 213,
 Jan. 1996, 247--258. Supplement, {\bf 65}, no. 213 (Jan. 1996) p. S15--S23.
 

\bibitem{VA}  A. V\'arilly-Alvarado, Arithmetic of del Pezzo surfaces, in
{\it Birational geometry, rational curves, and arithmetic} 
(F. Bogomolov, B. Hassett and Y. Tschinkel eds.) Simons Symposia {\bf 1} (2013), 293--319. 



\end{thebibliography}
\end{document}